\documentclass[11pt]{amsart}
\usepackage[top=3cm, bottom=3.5cm, left=3cm, right=3cm]{geometry}
\usepackage{amsmath,amsfonts,amssymb,amsthm}
\usepackage{mathtools}
\usepackage{tikz-cd}
\usepackage[utf8]{inputenc}
\usepackage{hyperref}
\usepackage{enumerate}
\usepackage[utf8]{inputenc} 
\usepackage{mathrsfs,amsmath}
\usepackage{amsfonts}
\usepackage{tikz-cd}

\usepackage{geometry} 
\usepackage{graphicx} 

 \usepackage[parfill]{parskip} 

\usepackage{booktabs} 
\usepackage{array} 
\usepackage{paralist} 
\usepackage{verbatim} 
\usepackage{subfig} 

\usepackage{fancyhdr} 
\newtheorem{thm}{Theorem}

\newtheorem{lemma}{Lemma}

\newtheorem{remark}{Remark}

\DeclareMathOperator\pr{pr}
\DeclareMathOperator\tr{tr} 
\DeclareMathOperator\Eucl{Eucl} 
\DeclareMathOperator\diam{diam}
\DeclareMathOperator\Bl{Bl}
\DeclareMathOperator\rad{rad}
\DeclareMathOperator\GH{GH}

\begin{document}
\bibliographystyle{amsplain}

\title{The K\"ahler-Ricci flow, holomorphic vector fields and Fano bundles}
\author{Xi Sisi Shen}
\date{}  
\maketitle
\begin{abstract}
We study the behavior of the K\"ahler-Ricci flow on compact manifolds developing finite-time singularities, in particular, when the flow contracts exceptional divisors or collapses the Fano fibers of a holomorphic fiber bundle. We present a technique using holomorphic vector fields to prove estimates related to the work of Song-Weinkove and Fu-Zhang.
\end{abstract}
\section{Introduction}
The Ricci flow, originally introduced by Hamilton \cite{hamilton82}, now serves as an essential tool in geometry. One of its most celebrated applications is in Perelman's ground-breaking resolution of the Poincar\'e and Geometrization conjectures \cite{perelman1,perelman3,perelman2} following Hamilton's program of Ricci flow with surgery. The Ricci flow starting from a K\"ahler metric remains K\"ahler along the flow and is thus referred to as the K\"ahler-Ricci flow. It was used by Cao \cite{cao85} to give a parabolic proof of the existence of K\"ahler-Einstein metrics on manifolds of negative first Chern class,  a result originally proved by Aubin \cite{aubin78} and Yau \cite{yau78}, and on manifolds of zero first Chern class, originally proved by Yau \cite{yau78}. There has since been extensive work done to better understand the behavior of the K\"ahler-Ricci flow in the general type case \cite{guo17,gsw16,tian-zhang06,tian-zhang16,tsuji88} in the Calabi-Yau fiber case \cite{fong-zhang15,gill14,song-weinkove12,song-tian07,twy14} as well as in the Fano case \cite{chen-tian06,chen-tian02,pssw09,pssw11,phong-sturm06,pssw08,sesum-tian08,szekelyhidi10,tian-zhu07,tian-zhu13,tzzz13,tian-zhang-fano}. Furthermore, the behavior of finite-time singularities along the K\"ahler-Ricci flow has been studied in \cite{fik03,fong14,fu-zhang17,song14,ssw13,st17,sw11,song-weinkove13,tian10} and references therein. 

The focus of this paper is on the behavior of the K\"ahler-Ricci flow on compact manifolds developing certain finite-time singularities. We consider, specifically, when the flow contracts exceptional divisors or collapses the Fano fibers of a holomorphic fiber bundle. We make use of an observation of \cite{song-weinkove13} that the quantity $|V|^2_{\omega(t)}$ satisfies a maximum principle for $V$ a holomorphic vector field and $\omega(t)$ a solution of the K\"ahler-Ricci flow. In the first part of the paper, we give an exposition of the estimates from \cite{song-weinkove13} along the flow as it contracts exceptional divisors. In the later sections, we consider the case where the flow collapses the fibers of a holomorphic fiber bundle where the fiber is $\mathbb{P}^m$ blown up at one point for $m\ge 2$, clarifying a result by \cite{fu-zhang17}. We prove that if the initial metric lies in a suitable K\"ahler class, then the flow collapses the fiber in finite time and the metrics converge along a subsequence in the Gromov-Hausdorf sense to a metric on the base.

We now describe our results more precisely. Let $(X,\omega_0)$ be a compact K\"ahler manifold of complex dimension $n\ge 2$, where $\omega_0 = \sqrt{-1}(g_0)_{i\bar{j}}dz^i \wedge d\bar{z}^j$ is the associated K\"ahler form to a K\"ahler metric $g_0$ on $X$. The K\"ahler-Ricci flow $\omega=\omega(t)$ satisfies
\begin{align*}
\frac{\partial}{\partial t}\omega = -\text{Ric}(\omega), \ \  \ \omega(0)=\omega_0
\end{align*}
where $\text{Ric}(\omega) = -\sqrt{-1}\partial\bar{\partial}\log \text{det} g$ and $g$ is the K\"ahler metric associated to $\omega$. From here on, we will also refer to $\omega$ as a K\"ahler metric.

It is well-known that the first singular time of the flow is given by
\begin{align*}
T = \text{sup}\{t>0 : [\omega_0]-Tc_1(X)>0\},
\end{align*}
where $c_1(X)=[\text{Ric}(\omega)]$ is the first Chern class of $X$ and $[\omega_0]-tc_1(X)$ is the K\"ahler class of $\omega(t)$.

Using the K\"ahler property and evolution equation of the K\"ahler-Ricci flow, we can bound the norm of holomorphic vector fields with respect to the evolving metric. Specifically, for a holomorphic vector field $V$ on $X$ and evolving metric $\omega(t)$, Song-Weinkove \cite{song-weinkove13} observed that
\begin{align}
(\partial_t-\Delta)\log(|V|^2_{\omega(t)})\le 0.\label{holoc_vf_bound}
\end{align}
This implies that the quantity $|V|^2_{\omega(t)}$ must achieve a maximum at $t=0$ or on the space boundary, simplifying the work to cases where we typically have more control. Bounds on $|V|^2_{\omega(t)}$ give us control of the evolving metric in the directions of $V$.

In Section 3, we will consider the behavior of the K\"ahler-Ricci flow on a manifold with disjoint, irreducible exceptional divisors arising from blowing up distinct points. An exceptional divisor arising from blowing up a point $y$ is given by the subvariety $E=\pi^{-1}(y)$ which is biholomorphic to $\mathbb{P}^{n-1}$ and represents all the directions through $y$, where we write $\pi:X\rightarrow Y$ for the blow-up map. Let $\omega(t)$ be a smooth solution to the K\"ahler-Ricci flow for $t\in [0,T)$ with the assumption that $T<\infty$ on a manifold $X$ with disjoint, irreducible exceptional divisors $E_1,\ldots, E_k$. Assume there exists a blow-up map $\pi: X\rightarrow Y$, where $(Y,\omega_Y)$ is a smooth compact manifold, such that $\pi(E_i)=y_i\in Y$ and that the initial K\"ahler class satisfies
\begin{align*}
[\omega_0]-2Tc_1(X) = [\pi^* \omega_Y].
\end{align*}
Song-Weinkove proved in \cite{song-weinkove13} that $\omega(t)$ converges in the Gromov-Hausdorff sense to a metric space $(Y,d_T)$ where $d_T$ is the metric extending $(\pi^{-1})^* g_T$ to be 0 on $\{y_1,\ldots,y_k\}$ and that $\omega(t)$ converges smoothly away from the exceptional divisors. In particular, they show that the diameter of $(X,\omega(t))$ is uniformly bounded for all $t\in [0,T)$. We present an exposition of their diameter estimate using local holomorphic vector fields, Equation \eqref{holoc_vf_bound}, and the fact that we have uniform estimates on compact subsets away from the divisor. While the proof of the diameter estimate in \cite{song-weinkove13} uses Equation \eqref{holoc_vf_bound} to obtain bounds on the radial holomorphic vector field for the purposes of bounding the lengths of radial paths emanating from the blow-up point, we will demonstrate that indeed Equation \eqref{holoc_vf_bound} can be used to show all of the crucial estimates. The main idea is to choose the holomorphic vector fields so that the bounds we get from the norm of these holomorphic vector fields are precisely in the directions that we need.

In Section 4, we will prove estimates for the case of collapsing the fibers of a Fano bundle. We consider the case where the fibers are projective space blown up at a point, $\Bl_p\mathbb{P}^m$, which are obtained from $\mathbb{P}^m$ by replacing a point $p\in\mathbb{P}^m$ with a subvariety $E$ biholomorphic to $\mathbb{P}^{m-1}$. We establish a diameter bound and convergence rate of the fibers, and show Gromov-Hausdorf convergence of the manifold with respect to the evolving metric subsequentially to the base manifold, building on results by \cite{fu-zhang17}. We require the fibers to be Fano to ensure that the K\"ahler-Ricci flow collapses the fiber in finite time. We begin by showing the necessary estimates when the manifold is a product manifold. From there, we demonstrate how our method for product manifolds can be rather straightforwardly adapted to the case of fiber bundles trivialized over Zariski open sets on a projective manifold in Section 5.

\begin{thm}
Let $\rho:X\rightarrow B$ be a Fano bundle trivialized over Zariski open sets on $B$, a compact, projective manifold, with Fano fiber $F \cong \Bl_p \mathbb{P}^m$.
Let $\omega(t)$ be a smooth solution to the K\"ahler-Ricci flow for $t\in [0,T)$ for $T<\infty$ on $X$. Assume that the initial K\"ahler class satisfies \label{mainthm}
\begin{align*}
[\omega_0]-2Tc_1(X) = [\rho^*\omega_B].
\end{align*}
Then the following hold:
\begin{enumerate}
\item The diameter of $X$ with respect to the evolving metric $\omega(t)$ is uniformly bounded for all $t\in [0,T)$.
\item There exists a uniform $C$ such that for any $t\in [0,T)$, 
\begin{align*}
\diam_{\omega(t)} F\le C(T-t)^{1/5},
\end{align*}
for every fiber $F$. 
\item $(X,\omega(t))$ converges subsequentially in the Gromov-Hausdorff sense to $(B,\omega_{B,\infty})$ where $\omega_{B,\infty}$ is uniformly equivalent to $\omega_B$.
\end{enumerate}
\end{thm}

\begin{remark}
Fu and Zhang assert a proof of the diameter bounds and convergence of the metrics in \cite{fu-zhang17} using similar arguments to those in  \cite{song-weinkove13}. However, in their proof of Lemma 3.3, their inequality in Equation (3.11) seems to implicitly assume that $[E]$ is a semi-negative line bundle.
\end{remark}

\begin{remark}
We remark that the exponent of $1/5$ appearing in the rate of convergence of the diameter of the fiber is a minor improvement to that shown in \cite{fu-zhang17} of $1/15$.
\end{remark}
 
In our proof, we require the existence of certain global holomorphic vector fields on the fiber manifold that extend to the whole manifold since we do not have uniform estimates away from the divisors and, thus, cannot simply work locally. We note that our method will allow us to bound each direction of the evolving metric explicitly rather than bounding the trace. 

\section{Preliminaries}
In this section, we cover several useful tools and establish the notation that we will use in later sections. The first key tool is the parabolic Schwarz lemma of Song-Tian \cite{song-tian07} (see also Theorem 2.6 of \cite{song-weinkove13}) which is a parabolic version of Yau's Schwarz lemma \cite{yau-schwarz}:
\begin{lemma}Let $F:X\rightarrow Y$ be holomorphic and let $\omega=\omega(t)$ be a solution to the K\"ahler-Ricci flow and $\omega_Y$ a fixed K\"ahler metric on $Y$, then we have for all points on $M\times [0,T)$ with $\tr_\omega F^*\omega_Y$ positive that
\begin{align*}
(\partial_t - \Delta)\log\tr_{\omega}F^*\omega_Y \le C\tr_{\omega}F^*\omega_Y,
\end{align*}
where $C$ is an upper bound of the bisectional curvature of $\omega_Y$.\label{psl}
\end{lemma}

We now state a well-known inequality following from a trace estimate along the flow due to Cao \cite{cao85} which is the parabolic version of an estimate of the complex Monge-Amp\`ere equation due to Aubin and Yau \cite{aubin78,yau78} (see also Proposition 2.5 of \cite{song-weinkove13}):
\begin{lemma}
Let $\omega=\omega(t)$ be a solution to the K\"ahler-Ricci flow and $\tilde{\omega}$ a fixed metric on $X$, then
\begin{align*}
(\partial_t - \Delta)\log\tr_{\tilde{\omega}}\omega \le C\tr_\omega \tilde{\omega},
\end{align*}
where $C$ depends only on a lower bound of the bisectional curvature of $\tilde{\omega}$.\label{ineq}
\end{lemma}
From these inequalities, we have the following well-known estimate due to Tian-Zhang \cite{tian-zhang06} and Zhang \cite{zhang06} (see also \cite{sw11,tian08} and Lemma 2.1 and 2.2 of \cite{song-weinkove13}):
\begin{lemma}
There exists a uniform constant $C$ depending only on $(X,\omega_0)$ such that the solution $\omega=\omega(t)$ to the K\"ahler-Ricci flow satisfies $\omega^n \le C\omega_0^n$.\label{volume_bound}
\end{lemma}
The following is another result we will need due to Tian-Zhang \cite{tian-zhang06} and Zhang \cite{zhang06} (cf. \cite{tsuji88}):
\begin{lemma}\label{metric_lower_bound}Given K\"ahler manifolds $(X,\omega_0)$ and $(Y,\omega_Y)$, if there exists a surjective holomorphic map $\zeta:X\rightarrow Y$ then for $\omega = \omega(t)$ a solution of the K\"ahler-Ricci flow on $X$ satisfying $\lim_{t\rightarrow T}[\omega(t)]=[\zeta^*\omega_Y]$ for $T<\infty$, there exists a uniform $c>0$ such that
\begin{align*}
\omega\ge c\zeta^*\omega_Y.
\end{align*}
\end{lemma}
In particular, the above lemma holds when $X$ is a fiber bundle over $Y$ (see Lemma 2.1 of \cite{ssw13}). 
\section{Contracting exceptional divisors}
Let $X$ be a compact K\"ahler manifold of dimension $n\ge 2$ with disjoint, irreducible exceptional divisors $E_1,\ldots, E_k$ arising from blowing up distinct points $p_1,\ldots,p_k$ and corresponding blow-up map $\pi: X\rightarrow Y$.
It has been shown in the work of \cite{song-weinkove13} that the K\"ahler-Ricci flow contracts exceptional divisors when the initial K\"ahler class satisfies
\begin{align*}
[\omega_0] - Tc_1(X) = [\pi^*\omega_Y].
\end{align*}
Song and Weinkove \cite{song-weinkove13} show in their paper that the diameter of $X$ is bounded along the flow and that the evolving metric converges in the Gromov-Hausdorff sense to a metric on $Y$ and smoothly away from the exceptional divisors. 
In this section, we provide a different point of view for computing one of the critical estimates in their paper,  Lemma 2.5(i) from \cite{song-weinkove13}, by replacing estimates on the traces of the evolving metric by estimates of the metric along certain directions given by holomorphic vector fields.

Since the divisors are disjoint, we may simply focus on the local behavior around a given exceptional divisor $E$ as we have uniform estimates on $\omega(t)$ on compact subsets away from the exceptional divisors. We will drop the variable $t$ from here on and denote the evolving metric by $\omega$. A proof of these uniform estimates was shown in \cite{tian-zhang06} using the fact that $\pi^*\omega_Y$ is uniformly equivalent to $\omega_0$ on $X\backslash \cup_{i=0}^m E_i$, coupled with the fact that $\omega^n\le C\omega_0^n$ and $\omega(t)\ge c\pi^*\omega_Y$ by Lemmas \ref{volume_bound} and \ref{metric_lower_bound}. By assumption, we have that $\pi(E)=p$ and $\pi^{-1}(p)\cong \mathbb{P}^{n-1}$. 

On the unit ball $D\subset Y$ around $p$, define coordinates $z^1,\ldots, z^n$ which we can pull-back via $\pi$ to $X$. We can identify $\pi^{-1}(D)$ with the submanifold 
\begin{align*}
\tilde{D}=\{(z,l)\in D\times \mathbb{P}^{n-1}| z^pl^q = z^ql^p\}.
\end{align*}
Now, let us define for each pair of values $i,j\in \{1,\ldots, n\}$, a holomorphic vector field $$V^k_\ell = z^\ell \frac{\partial}{\partial z^k},$$ where we note that $k,\ell$ are not summation variables, which then defines via $\pi$ a holomorphic vector field on $\pi^{-1}(D)$ vanishing to order 1 along $E$. We may extend each $V^k_\ell$ to a smooth (not holomorphic) global vector field on the whole of $X$. In the following lemma, we are identifying $\pi^{-1}(D\backslash \{0\})$ with $D\backslash \{0\}$ via the map $\pi$ and writing $\omega$ for the K\"ahler metric $(\pi^{-1})^*\omega$ on $D\backslash \{0\}\subset Y$. Since we have bounds on compact subsets away from $E$, it suffices to prove bounds on $\pi^{-1}(D)$.

In order to arrive at a uniform diameter bound of $X$ with respect to the evolving metric $\omega(t)$, we will need to uniformly bound the following: (1)  the diameter of spheres centered at the blow-up point and (2) the length of radial paths emanating from the blow-up point. We begin by proving the first of these two bounds:

\begin{lemma}
For $r\in (0,1)$, the diameter of the $2n-1$ sphere $S_r$ of radius $r$ in $D$ centered at the blow-up point with respect to the metric induced from $\omega$ is uniformly bounded from above, independent of $t$ and $r$.
\label{sphere_diameter_bound}
\end{lemma}
\begin{proof}
We begin by showing that on $\pi^{-1}(D)$, there exists a constant $C$ such that 
\begin{align*}
g_{\mathbb{R}}\big(\frac{\partial}{\partial x^1},\frac{\partial}{\partial x^1}\big) = 2g_{1\bar{1}}\le \frac{C}{r^2}
\end{align*}
for the Euclidean metric $g_{\mathbb{R}}$ on $D$ and $r^2 =\sum_{k=1}^{n}|z^k|^2$.
We proceed by an application of the maximum principle and by considering the holomorphic vector fields $V^1_\ell$ for $\ell=1,\ldots,n$. 
Note that for each fixed $t$, we have that $g_{i\bar{j}}$ is uniformly equivalent (with a constant that depends on $t$) to the model metric 
\begin{align}\begin{split}
\pi^*\omega_Y + \tau^*\omega_{FS} = \delta_{ij}+\frac{\sqrt{-1}}{2\pi r^2}\big(\delta_{ij}-\frac{\overline{z^i} z^j}{r^2}\big)\label{model_metric}
\end{split}\end{align}
where $\delta_{ij}$ is the pullback of the Euclidean metric on $Y$ by $\pi$ and the second term is the pullback of the Fubini-Study metric  on $\mathbb{P}^{n-1}$ by $\tau: \tilde{D}\rightarrow \mathbb{P}^{n-1}$ taking $(z,l)\mapsto l$ where $\frac{l^p}{l^q}=\frac{z^p}{z^q}$. 

For each fixed $t$, we have that
\begin{align*}
|V^1_\ell|^2_{\omega} = g_{1\bar{1}}(t)|z^\ell|^2\le C(t) \Big(1+\frac{\sqrt{-1}}{2\pi r^2}\big(1-\frac{|z^1|^2}{r^2}\big)\Big)|z^\ell|^2\le C(t) \ \text{on} \  \pi^{-1}(D),
\end{align*}
and is a smoothly defined quantity. It follows from Equation \eqref{holoc_vf_bound}, that on $\pi^{-1}(D)$,
\begin{align*}
(\partial_t-\Delta)\log|V^1_\ell|^2_\omega &\le 0
\end{align*}
for all $\ell=1,\ldots,n$. This implies that a maximum occurs either on the boundary of $\pi^{-1}(D)$ or at $t=0$. 
Since we have uniform estimates on compact subsets away from the exceptional divisor, we have that $|V^1_\ell|^2_\omega\le C$ on the boundary of $\pi^{-1}(D)$. If the maximum occurs at $t=0$ then 
\begin{align*}
|V^1_\ell|^2_\omega \le &\sup_{\pi^{-1}(D)} |V^1_\ell|^2_{\omega_0}\le C .
\end{align*}
By the definition of $|V_1^\ell|^2_\omega$, this gives us that
\begin{align*}
\ g_{1\bar{1}}&\le \frac{C}{|z^\ell|^2} \ \  \text{for } \ell=1,\ldots,n \ \ 
\Rightarrow \ \ g_{1\bar{1}}\le \frac{C}{\max_\ell|z^\ell|^2} \ \
\Rightarrow \ \ g_{1\bar{1}}\le \frac{C}{r^2}
\end{align*}
An identical argument can be used to obtain that $g_{k\bar{k}}\le \frac{C}{r^2}$ for $k\in \{2,\ldots, n\}$. Together, this gives us that on $\pi^{-1}(D)$,
\begin{align}\begin{split}
\omega\le \frac{C}{r^2}\omega_{\Eucl} .\label{lemma1_bd}
\end{split}\end{align}

From this, we are left to show bounds on the diameter of the $2n-1$ sphere $S_r$ for $r\in(0,1)$ centered at the blow-up point which we may identify with its preimage under $\pi$, independent of $r$. Consider the inclusion map $\iota_r: S_r\rightarrow D$. Using Equation \eqref{lemma1_bd}, we have for any $r\in (0,1)$ and $p, q\in S_r$ that
\begin{align*}
d_{\iota_r^* \omega(t)}(p,q)\le \frac{\sqrt{C}}{r}d_{\iota_r^* \omega_{\Eucl}}(p,q)\le \sqrt{C}\pi
\end{align*}
which is independent of $t$ and $r$. 
\end{proof}

To complete the proof of the diameter bound of $X$ with respect to $\omega(t)$, we will need to combine this result with a bound on the length of radial paths in $D\backslash \{0\}$. Define the radial vector field $V^{\rad}= \sum_{\ell=1}^n z^\ell\frac{\partial}{\partial z^\ell}$ on the unit ball centered at the blow-up point in $Y$, which we may identify via $\pi$ to be a holomorphic vector field on $X$ in a neighborhood of $D$.\\

\begin{lemma}
For any $x\in\pi^{-1}(D\backslash\{0\})$, the length of the radial path $\gamma(\lambda)=\lambda x$ for $\lambda\in (0,1]$ with respect to $\omega$ is uniformly bounded from above by a uniform constant multiple of $|x|^{1/2}$. \label{radial_path_bound}
\end{lemma}
\begin{proof}
We will show that on the complex line $\{0=z^1=\ldots=\widehat{z^i}=\ldots=z^n\}$,
\begin{align*}
g_{i\bar{i}}\le \frac{C}{|z^i|} 
\end{align*}
where for the purposes of demonstration we consider the case $i=1$. Define the quantity $Q_\varepsilon = \log(|V^{\rad}|^{2(1+\varepsilon)}_\omega\left|\frac{\partial}{\partial z^1}\right|^2_\omega) $ on $D$. For a fixed $t$, it can be shown that $|V^{\rad}|^2_\omega$ is uniformly equivalent to $r^2$ in $D$ by a straightforward computation using the model metric in Equation \eqref{model_metric}. By the previous lemma, $g_{1\bar{1}}\le \frac{C}{r^2}$ implies that $Q_\varepsilon$ tends to $-\infty$ as $x\rightarrow E$. Thus, a maximum must occur on $D\backslash\{0\}$ on which we have that
\begin{align*}
(\partial_t - \Delta)Q_\varepsilon = (1+\varepsilon)(\partial_t - \Delta)\log|V^{\rad}|^2_\omega+(\partial_t-\Delta)\log \left|\tfrac{\partial}{\partial z^1}\right|^2_\omega \le 0
\end{align*}
where the inequality follows from Equation \eqref{holoc_vf_bound} since both $V^{\rad}$ and $\frac{\partial}{\partial z^1}$ are holomorphic vector fields on $X$ away from the origin. Letting $\varepsilon\rightarrow 0$, this gives us the following bound
\begin{align*}
|V^{\rad}|^2_\omega \left|\tfrac{\partial}{\partial z^1}\right|^2_\omega \le C 
\end{align*}
Restricting to the complex line $\{0=z_2=\ldots z_n\}$, we have that
\begin{align*}
g_{1\bar{1}}|z^1|^2 g_{1\bar{1}}\le C  \ \ \Rightarrow \ \ g_{1\bar{1}}\le \frac{C}{|z^1|}.
\end{align*}
The same argument can then be used for all $i=2,\ldots, n$, to get that on $\{0=z_1=\ldots=\widehat{z}_i=z_n\}$ we have
\begin{align*}
g_{\mathbb{R}}\big(\frac{\partial}{\partial x^i},\frac{\partial}{\partial x^i}\big)=2g_{i\bar{i}}\le \frac{C}{|z^i|}
\end{align*}
Without loss of generality, we may assume that $x$ lies in a coordinate direction from the origin. For the path $\gamma(\lambda)=\lambda x$ for $\lambda\in (0,1]$ emanating from the blow-up point, we can show that 
\begin{align*}
g_{\gamma(\lambda)}(\gamma'(\lambda),\gamma'(\lambda))\le \frac{C|x|}{\lambda}
\end{align*}

whose square root when integrated over $\lambda\in (0,1]$ gives us the desired bound on the length of $\gamma$.
\end{proof}
Combining Lemma \ref{sphere_diameter_bound} and Lemma \ref{radial_path_bound}, we are ready to prove the following result by \cite{song-weinkove13}:

\begin{thm}
Let $\omega(t)$ be a smooth solution to the K\"ahler-Ricci flow for $t\in [0,T)$ and assume $T<\infty$ on a manifold $X$ with disjoint exceptional divisors $E_1,\ldots, E_k$. Assume there exists a blow-up map $\pi: X\rightarrow Y$ for $(Y,\omega_Y)$ a compact manifold such that $\pi(E_i)=y_i\in Y$ and that the initial K\"ahler class satisfies
\begin{align*}
[\omega_0]-2Tc_1(X) = [\pi^* \omega_Y] 
\end{align*}
Then the diameter of $X$ with respect to the evolving metric $\omega(t)$ is uniformly bounded for all $t\in [0,T)$.
\end{thm}
\begin{proof}
From Lemma \ref{sphere_diameter_bound} and Lemma \ref{radial_path_bound}, we have a uniform bound on the diameter of $X$ when $p,q$ do not lie on $E$ since any two points in $\pi^{-1}(D\backslash \{ 0\})$ are connected by a combination of radial paths in coordinate directions and walks along spheres of radius less than 1. Let us assume without loss of generality that $p$ lies on the exceptional divisor $E$. We can take limiting sequences $(p_j)_{j=1}^\infty$ such that $p_j\rightarrow p$ with respect to any fixed metric and $p_j \in D\backslash \{0\} \ \forall j=1,\ldots,\infty$. Since $d_{\omega}(p_j,q)\le C$ for a uniform $C$ independent of $j$, we may conclude that this holds for all $j$ and thus indeed holds for $p,q$.
\end{proof}

We note a few key differences between our proof and that of \cite{song-weinkove13}. In bounding the diameters of spheres, their method bounds the trace using a maximum principle argument involving controlling bad terms using the K\"ahler potential. In contrast, our proof computes bounds on the evolving metric in each direction using the holomorphic vector fields that we have chosen. In bounding the lengths of radial paths, \cite{song-weinkove13} uses the radial holomorphic vector field $V^{\rad}=z^i\frac{\partial}{\partial z^i}$ coupled with a trace term. We use the same radial holomorphic vector field but pair it with the norm of the holomorphic vector field $\frac{\partial}{\partial z^j}$ in lieu of the trace to obtain the bound in each of the $z^j$ directions pointing away from the blow-up point.

\section{Collapsing Fano component of product manifold}
In this section, we prove Theorem 1 in the product case. Let $X=B\times F$ where $B$ is a compact K\"ahler manifold of dimension $n-m$ and $F$ is a compact Fano manifold of dimension $m$. We are working in the setting where we collapse the entire fiber component $F$ of $X$ along the K\"ahler-Ricci flow. In the next section, we will describe how this can be generalized to a compact manifold whose projective base manifold is covered by Zariski open sets over which fibers are trivial.

Assume that $F\cong \Bl_p\mathbb{P}^{m}$ which is a Fano manifold and so will be collapsed by the K\"ahler-Ricci flow in finite time. Let $\pi: F\rightarrow \mathbb{P}^m$ be the blow-up map and assume that the initial K\"ahler class satisfies
\begin{align*}
[\omega_0]-Tc_1(X)=[\pr^*_B\omega_B]
\end{align*}
where $\pr_B : B\times F \rightarrow B$ is the projection onto the base component. In addition, we assume that the initial metric $\omega_0$ is not a product metric since otherwise the flow will deform the metrics on $B$ and $F$ independently until the volume of $F$ vanishes. 
We note that unlike in the previous case, we no longer have uniform estimates of the metric on $X$ away from the exceptional divisors on the fiber manifold. For this reason, we will use the existence of certain global holomorphic vector fields on $F$ to clarify the estimates of \cite{fu-zhang17}.

For the rest of the section, we will be working ``downstairs" in $\mathbb{P}^m$ via the blow-up map $\pi:F\rightarrow \mathbb{P}^m$, and for simplicity refer to $(\pi^{-1})^* \omega$ by $\omega$ and $\pi^{-1}(D_r)$ by $D_r$ . We have on each coordinate chart $$U_k=\{[Z_0:\ldots:Z_{m}]| Z_k\neq 0\}$$ on $\mathbb{P}^{m}$ for $k=0,\ldots, m$, local coordinates $w^1,\ldots,w^m$ where $w^i = \frac{Z_{i-1}}{Z_k}$ on $U_k$ for $1\le i\le k$ and $w^i = \frac{Z_i}{Z_k}$ for $k<i\le m$. Without loss of generality may assume that the blow-up point is at the point $p=[1:0:\ldots:0]$.

Let us define how we will restrict the metric to the base and fiber components. Define the \emph{restriction} of a metric $\omega$ on $B\times F$ to $F$, denoted $\omega|_F$, at each point $(b,f)\in B\times F$ by 
\begin{align}\label{restriction}
\omega|_F(V,\overline{W}) := \omega_{(b,f)}((\iota_{f})_*V,(\iota_{f})_*\overline{W})
\end{align}
where $\iota: F\xhookrightarrow{} B\times F$ given by $\iota_b(f)=(b,f)$ and $V,W$ are vector fields on $F$. We can define $\omega|_B$ in an analogous way.

\subsection{Diameter bound of X}

We now describe how we obtain a diameter bound on $X$ with respect to the evolving metric $\omega=\omega(t)$. Firstly, we will bound the diameter of $B$ by proving that
\begin{align}
\omega|_B\le C(\pr^*_B\omega_B)|_B \label{base_fano_bound},
\end{align}
where $\omega|_B$ is defined as in Equation \eqref{restriction}.

Secondly, to bound the diameter of $F$, we choose $r_0$ sufficiently large such that $\cup_{i=0}^m D^i_{r_0}$ form a cover of $\mathbb{P}^m$, where
\begin{align*}
D^i_{r_0}=\{(\tfrac{Z_0}{Z_i},\ldots, \widehat{\tfrac{Z^i}{Z_i}},\ldots, \tfrac{Z^m}{Z_i}): \sum_{k=0,k\neq i}^m \big|\tfrac{Z_k}{Z_i}\big|^2 \le r_0^2\} \subset U_i.
\end{align*}
We will show the diameter bound on $F$ by showing that on $D^0_{r_0}$, the ball centered around the blow-up point in the coordinates of $U_0$, we have that
\begin{align}
\omega|_F\le \frac{C}{r^2}(\pr^*_F\pi^*\omega_{\mathbb{P}^m})|_F, \label{trace_fano_bound}
\end{align}
where $\text{pr}_F: B\times F \rightarrow F$ is the projection onto the fiber component and $\omega|_F$ is defined as in Equation \eqref{restriction}.

On $D^k_{r_0}$ for $1\le k\le m$ which do not contain the blow-up point, we will show that
\begin{align}
\omega|_F\le C(\pr^*_F\pi^*\omega_{\mathbb{P}^m})|_F \label{trace_fano_no_bp_bound}.
\end{align}

\subsubsection{Diameter bound of $B$}
From the fact that the limiting cohomology class has a component on $B$, we may apply Lemma \ref{metric_lower_bound} to get that $\omega\ge C\pr^*_B\omega_B$. Restricting to $B$ this gives us that 
\begin{align}
\omega|_B\ge C(\pr^*_B\omega_B)|_B .\label{base_bound}
\end{align}
We will prove Equation \eqref{base_fano_bound} by showing a bound on horizontal level sets $H$ on $X$:

\begin{lemma}\label{base_control}
For any $f\in F$ and corresponding level set $H=B\times \{f\}$ we have that 
\begin{align*}
\omega|_H \le C (\pr^*_B\omega_B)|_H.
\end{align*}
\end{lemma}
\begin{proof}
For fixed $f\in F$, we wish to bound the quantity $Q=\text{tr}_{(\text{pr}_B^* \omega_B)|_H} \omega|_H$. Assume that the maximum occurs at a point $x_0=(b_0,f)$ and let us choose normal coordinates $z^1,\ldots, z^n$ for $(\pr_B^*\omega_B)|_H$ at $x_0$ for which $g$ is diagonal where $z^1,\ldots, z^{n-m}$ are coordinates on $B$. Then we can write $Q = \sum_{i,j=1}^{n-m}g_B^{i\bar{j}}g_{i\bar{j}}$. A straightforward computation following similarly to the proof of Lemma \ref{ineq} shows that
\begin{align*}
(\partial_t - \Delta) \log Q &= \frac{1}{Q}\sum_{a,b,i,j=1}^{n-m}\Big(-g^{a\bar{b}}(R_B)_{a\bar{b}}^{\ \ i\bar{j}}g_{i\bar{j}} + g^{a\bar{b}}g^{p\bar{q}}(\partial_a g_{i\bar{q}})(\partial_{\bar{b}}g_{p\bar{j}})g_B^{i\bar{j}} + \frac{|\nabla (g_B^{i\bar{j}}g_{i\bar{j}})|_\omega^2}{Q}\Big) \\
& \le -\frac{1}{Q}\sum_{a,b,i,j=1}^{n-m}g^{a\bar{b}}(R_B)_{a\bar{b}}^{\ \ i\bar{j}}g_{i\bar{j}} \\
& \le C\tr_{\omega|_H}(\pr_B^*\omega_B)|_H,
\end{align*}
where $C$ depends on a lower bound for the holomorphic bisectional curvature of $\omega_B$ and the second last inequality uses two applications of the Cauchy-Schwarz inequality.
Using Equation \eqref{base_bound}, we have that
\begin{align*}
(\partial_t - \Delta)\log Q \le C\text{tr}_{\omega|_H}(\pr^*_B\omega_B)|_H\le C.
\end{align*}
Since $T<\infty$, $(\partial_t-\Delta)(\log Q - Ct)\le 0$ which implies that $Q\le C$, giving us the desired inequality.
\end{proof}
By the compactness of $F$, the above lemma implies Equation \eqref{base_fano_bound}.

\subsubsection{Diameter bound of spheres centered around E}
\begin{lemma}\label{bounds_on_spheres}
There exists a uniform $C$ such that for each point $(b,f)\in B\times F$ 
$$\omega|_F\le \max(\frac{C}{r^2}, C)(\pr^*_F\pi^*\omega_{\mathbb{P}^m})|_F ,$$
where $r^2=\big|\frac{Z_1}{Z_0}\big|^2+\big|\frac{Z_2}{Z_0}\big|^2$ is the distance squared to the blow-up point.
\end{lemma}

\begin{proof}
To achieve a bound on $F$, we may work ``downstairs" on $\mathbb{P}^m$. We will bound the evolving metric in local coordinates on each coordinate chart. We begin by proving the desired bounds on $U_0$, the only chart that contains a blow-up point. It is well known that $\dim(H^0(\mathbb{P}^m,T^{1,0}\mathbb{P}^m))=(m+1)^2-1$ by the fact that $\text{Aut}(\mathbb{P}^m) = \text{PGL}(m+1,\mathbb{C})$. It can be shown that $H^0(\mathbb{P}^m,T^{1,0}\mathbb{P}^m)$ is spanned by, written in the local coordinates $w^1,\ldots, w^m$ of $U_0$,  $m$ holomorphic vector fields of the form $$V^i = \frac{\partial}{\partial w^i},$$ for $i=1,\ldots, m$, followed by $m^2$ holomorphic vector fields of the form $$V^i_j=w^j\frac{\partial}{\partial w^i},$$ for $i,j=1,\ldots, m$ and $m$ holomorphic vector fields of the form $$V^{\rad}_i = w^i\sum_{\ell=1}^m w^\ell \frac{\partial}{\partial w^\ell},$$ for $i=1,\ldots, m$.

Using the fact that a holomorphic vector field on $M$ vanishing at a point $p$ extends to a holomorphic vector field on $\Bl_p M$, the above-defined vector fields $V^i_j$ and $V^{\rad}_i$ for all $i,j=1,\ldots,m$ lift to $F\cong \Bl_p\mathbb{P}^m$ since they each vanish at the blow-up point $p=[1:0:\ldots:0]$. From there, each of these holomorphic vector fields can then be trivially extended to all of $X$. 

Since $|V^i_j|^2_\omega$ is a globally-defined smooth quantity on $X$ and $V^i_j$ is holomorphic, we have by Equation \ref{holoc_vf_bound} that
\begin{align*}
(\partial_t - \Delta)\log(|V^i_j|^2_\omega) \le 0.
\end{align*}
Using the maximum principle, a maximum occurs at $t=0$, and so $|V^i_j|^2_\omega \le C|V^i_j|^2_{\omega_0}$ which implies that $g_{i\bar{i}}\le \frac{C}{|z^j|^2}$. Since this holds for each pair $(i,j)$, it follows that $g_{i\bar{i}}\le \frac{C}{r^2}$ for all $i=1,\ldots, m$. On $D^0_{r_0}$, this gives us that $\omega|_F \le \frac{C}{r^2}(\pr^*_F\pi^*\omega_{\mathbb{P}^m})|_F$.

On each $U_k$ for $1\le k\le m$, it can be straightforwardly checked that $V^{\rad}_k$ defined above on $U_0$ transforms onto the local coordinates $w^1,\ldots,w^m$ of $U_k$ as $\frac{\partial}{\partial w^1}$. In addition, $V^i_k$ transforms onto $U_k$ as $\frac{\partial}{\partial w^{i+1}}$ for $i<k$ and as $\frac{\partial}{\partial w^i}$ for $i>k$. Since $V^{\rad}_k$ can be lifted to $\Bl_p\mathbb{P}^m$ and trivially extended to $X$, by Equation \ref{holoc_vf_bound}, we have that
\begin{align*}
(\partial_t - \Delta)\log(|V^{\rad}_k|^2_\omega)\le 0
\end{align*}
and, thus, a maximum must occur at $t=0$ giving us that $|V^{\rad}_k|^2_\omega\le C$ and so on $U_k$ this implies that $g_{1\bar{1}}\le C$. Since we have already previously shown that $|V^i_k|^2_\omega\le C$, this gives us that on $U_k$, $g_{i\bar{i}}\le C$ for all $i=2,\ldots, m$. Together, this gives us that on $D^k_{r_0}$ that $\omega|F\le C\pr^*_F\pi^*\omega_{\mathbb{P}^m}$.
\end{proof}

We note that our proof differs from the one used by \cite{fu-zhang17} in several ways. Firstly, although we use Lemma \ref{base_control} to control the bounds on the base in the same way, our estimates on the fiber completely rely on the existence of the specific holomorphic vector fields instead of using sections vanishing along the exceptional divisors. Reiterating what we mentioned in the previous section, our method provides estimates on the metric in the directions of particular holomorphic vector fields instead of bounding the trace with respect to a singular metric. In this way, we are not using the positivity of the bisectional curvature of the Fubini-Study metric on $\mathbb{P}^m$ but are instead using the symmetries of $\mathbb{P}^m$ and the global holomorphic vector fields that it permits.

\subsubsection{Bounds on lengths of radial paths}
Combining what we showed in the previous section, that $\omega \le C\omega_0$ on each $D_{r_0}\backslash D_{\delta}$, and that $\omega^n\le C\omega_0^n$ by Lemma \ref{volume_bound}, we indeed have uniform estimates for $\omega$ away from the exceptional divisors. Showing bounds on the lengths of radial paths follows straightforwardly from the argument of \cite{song-weinkove13} but we include a proof here using holomorphic vector fields.

\begin{lemma}
The length of the radial path on the fiber emanating from the blow-up point to a point $x$ on the same coordinate chart $\gamma(\lambda)=\lambda x$ for $\lambda\in(0,1]$ with respect to $\omega$ is bounded by a uniform constant multiple of $|x|^{1/2}$.\label{lengths_of_radial_paths}
\end{lemma}
\begin{proof}
As in the last section, we will prove that on the complex line $$\{0=z^1=\ldots=\widehat{z^i}=\ldots=z^n\},$$ we have
\begin{align*}
g_{i\bar{i}}\le \frac{C}{|z^i|}.
\end{align*}
Let us consider the case $i=1$, as the estimates for $i=2,\ldots,n$ follow precisely in the same manner.
We will use the radial vector field we introduced in the proof of the previous theorem defined on each chart's local coordinates $w^1,\ldots,w^m$ by $V^{\rad}= \sum_{\ell=1}^{m}w^\ell \frac{\partial}{\partial w^\ell}$. Since we have bounds away from the exceptional divisors from the previous section, we will work only on $D_1$. Consider the quantity $$Q_\varepsilon = \log(|V^{\rad}|^{2(1+\varepsilon)}_{\omega}\left|\tfrac{\partial}{\partial z^1}\right|^2_{\omega}).$$ For fixed $t$, $Q_\varepsilon\rightarrow -\infty$ as $x\rightarrow E$ since $\left|\frac{\partial}{\partial z^1}\right|^2_\omega \le \frac{C(t)}{r^2}$ and $|V^{\rad}|^2_{\omega}\le C(t)r^2$. This implies that a maximum cannot occur on $E$. Using the fact that $V^{\rad}$ and $\frac{\partial}{\partial z^1}$ are holomorphic vector fields, we know that $(\partial_t - \Delta)Q_\varepsilon \le 0$. This gives that $Q_\varepsilon \le C$ for all $\varepsilon$ since we have bounds on $\partial D_1$. Letting $\varepsilon\rightarrow 0$ we have 
\begin{align}
|V^{\rad}|^2_\omega\left|\tfrac{\partial}{\partial z^1}\right|^2_\omega\le C .
\end{align}
Restricting to the complex line $\{0=z^1=\ldots=\widehat{z^i}=\ldots=z^n\}$, we arrive at
\begin{align}
g_{1\bar{1}}|z^1|^2 g_{1\bar{1}}\omega \le C.
\label{vr}
\end{align}
From this we obtain the desired estimate
$$g_{i\bar{i}}\le \frac{C}{|z^i|}.$$
From the above inequality, it is straightforward to see that if we assume that $x$ lies in a coordinate direction from the origin, $\gamma(\lambda)=\lambda x$ for $\lambda\in(0,1]$, then
$$g_{\gamma(\lambda)} (\gamma'(\lambda),\gamma'(\lambda)) \le \frac{C|x|}{\lambda}.$$ Integrating its square root over $\lambda\in[0,1)$, we arrive at the desired bound on radial paths emanating from a blow-up point.
\end{proof}

Combining Lemma \ref{bounds_on_spheres} and Lemma \ref{lengths_of_radial_paths}, we may conclude a diameter bound for the fiber $F$. Lemma \ref{base_control} automatically gives us bounds on the diameter of $B$ for each $f\in F$. Thus, this gives a uniform bound on the diameter of $X$ with respect to $\omega$ is for all $t\in [0,T)$.

\subsection{Convergence of the diameter of fibers}
We will now show that the diameter of the fiber tends to zero as $t\rightarrow T^{-}$ at a rate of $(T-t)^{1/5}$. We note that the proofs of the lemmas follows similarly to those in \cite{song-weinkove13} and \cite{fu-zhang17}, but we will include a proof for the convenience of the reader. The improvement in the exponent from the result of \cite{fu-zhang17} of $1/15$ to $1/5$ is the result of simply adjusting the powers of certain parameters in the proof accordingly. From now on, let us denote $d_t$ for $d_{\omega(t)}$ and $\diam_t$ for $\diam_{\omega(t)}$.

We will use the following method of \cite{song-weinkove13} (see also \cite{ssw13}):

\begin{lemma}
Given two points $p,q$ on $F$ that can be joined by a curve $\gamma \cong \mathbb{P}^1\subset F$, assume that $p,q$ belong to the same fixed coordinate chart $U$ whose image under the holomorphic coordinate $z=x+\sqrt{-1}y$ is a ball of radius 2 in $\mathbb{C}$ with respect to $\omega_{\Eucl}$. Define $$\mathcal{R}=\{(x,y)\in \mathbb{R}^2: x\in [0,x_0], y\in [-\varepsilon,\varepsilon]\}\subset \mathbb{R}^2\subset \mathbb{C}$$ where $\varepsilon=(T-t)^\alpha$ which we may assume is sufficiently small with $\alpha>0$. Assume that $p$ corresponds to $(0,0)$ and $q$ to $(x_0,0)$ with $x_0\in(0,1)$ in $\mathcal{R}$. Then for a specific $y'\in (-\varepsilon, \varepsilon)$, we have that for $p' = (0,y')$ and $q'=(x_0,y')$ that
\begin{align*}
d_t(p',q') \le C(T-t)^{\frac{1-\alpha}{2}}.
\end{align*}
\label{rectangle}
\end{lemma}
\begin{proof}
Along this $\gamma$, we have $$\int_\gamma \omega(t) \le \int_\gamma \frac{1}{T}((T-t)\omega_0 + t\pr_B^*\omega_B) \le C(T-t).$$ Then it follows that $$\int_{-\varepsilon}^\varepsilon \int_0^{x_0}\tr_{\tilde{\omega}}\omega \ dx dy\le C(T-t)$$ which implies that there exists $y'\in(-\varepsilon,\varepsilon)$ such that $$\int_0^{x_0} \tr_{\tilde{\omega}}\omega(t) \ dx \le \frac{C(T-t)}{\varepsilon}=C(T-t)^{1-\alpha}$$ since $\varepsilon = (T-t)^\alpha.$ For $p'\in(0,y')$, $q'\in(x_0,y')$, we have that 
\begin{align*}
d_t(p',q') &\le \int_0^{x_0} \sqrt{g_t(\partial_x,\partial_x)}(x,y')dx\\
&= \int_0^{x_0}\sqrt{\tr_{\tilde{\omega}}\omega(t)}\sqrt{\widehat{g_0}(\partial_x,\partial_x)}(x,y')dx\\
&\le \Big(\int_0^{x_0}\tr_{\tilde{\omega}}\omega(t) \ dx \Big)^{1/2}\Big(\int_0^{x_0}\widehat{g_0}(\partial_x,\partial_x)(x,y')dx \Big)^{1/2}\\
&\le C(T-t)^{\frac{1-\alpha}{2}},
\end{align*}
as desired.
\end{proof}

\begin{lemma}
Given two points $p,q\in E\subset F$, we have that $$d_t(p,q)\le C(T-t)^{1/3}.$$
\end{lemma}
\begin{proof}
We first apply the previous Lemma with $\alpha=1/3$ , i.e. $\varepsilon = (T-t)^{1/3}$, to obtain that on the rectangle $\mathcal{R}$ as specified in the Lemma, we have that $d_t(p',q')\le C(T-t)^{1/3}$. Now, it remains to show how to bound $d_t(p,p')$ and $d_t(q,q')$. The points $p$ and $p'$ correspond to lines $L_{p}$ and $L_{p'}$ in $\mathbb{P}^1$ which when restricted $D_r\subset\mathbb{C}^2$ become paths emanating from the origin to points $\tilde{p}$ and $\tilde{p}'$ that are transverse to $E$. Now, by the bound from Lemma \ref{lengths_of_radial_paths} this gives us that $d_t(p,\tilde{p})+d_t(p',\tilde{p}')\le Cr^{1/2}\le C\varepsilon$ if we choose $r$ sufficiently small. Finally, since $p$ and $p'$ are $\varepsilon$ apart on $\mathcal{R}$ which implies that that $\tilde{p}$ and $\tilde{p}'$ are $r\varepsilon$ apart on $S_r$ with respect to the Euclidean metric, it follows that for $\iota_r:S_r\rightarrow D_r$, we have $$d_t(\tilde{p},\tilde{p}')\le \frac{\sqrt{C}}{r}d_{\iota^*_r\omega_{\Eucl}}(\tilde{p},\tilde{p}')\le C\varepsilon=C(T-t)^{1/3}.$$ Applying the same argument for $q$ and $q'$ and applying the triangle inequality, we arrive at the desired bound.
\end{proof}

\begin{lemma}
There exists a uniform constant $C$ such that for any fixed $b_0\in B$, $\delta_0\in(0,1/2)$ and $t\in [0,T)$, we have
\begin{align*}
\diam_t(\{b_0\}\times \pi^{-1}(D_{\delta_0}))\le C(|\delta_0|^{1/2}+(T-t)^{1/3}).
\end{align*}
\end{lemma}
\begin{proof}
We may assume that at least one of $p,q\in \pi^{-1}(D_{\delta_0})$ is not in $E$ since otherwise it follows by the previous lemma. For $p\in E$ and $q\in \pi^{-1}(D_{\delta_0}\backslash \{0\})$, let $\gamma(\lambda)=\lambda q$ for $\lambda\in [0,1]$ be a radial path in $\pi^{-1}(D_{\delta_0})$ from $q$ to $E$. Let $q' = \lim_{\lambda\rightarrow 0^+}\gamma(\lambda)$. Then, by Lemma \ref{lengths_of_radial_paths}, $d_t(q,q')\le C|\delta_0|^{1/2}$ and by the previous lemma $d_t(p,q')\le C(T-t)^{1/3}$.
\end{proof}

\begin{lemma}
There exists a uniform constant $C$ such that for any fixed $b_0\in B$ and any $p,q\in F_{b_0}$ and $t\in [0,T)$, we have
\begin{align*}
\diam_t(F_{b_0})\le C(T-t)^{1/5}.
\end{align*}\label{fiber_convergence}
\end{lemma}
\begin{proof}
Using the previous lemma, we see that for $\delta_0 = (T-t)^{2/5}$, we have that $$\diam_t(\{b_0\}\times\pi^{-1}(D_{\delta_0}))\le C(T-t)^{1/5}.$$ Now, it remains to consider the case when $p,q\in F\backslash\pi^{-1}(D_{\delta_0})$. Applying Lemma \ref{rectangle} to $p,q$ with $\varepsilon = (T-t)^{3/5}$, we have that $d_t(p',q')\le C(T-t)^{1/5}$. We also have by Lemma \ref{bounds_on_spheres} that $$\omega|_F \le \frac{C}{\delta_0^2}\pi^*\omega_{FS}\le \frac{C}{\delta_0^2}\pi^*\omega_{\Eucl},$$ which gives us that $d_t(p,p')\le \frac{\sqrt{C}}{\delta_0}\varepsilon \le C(T-t)^{1/5}$ and similary for $d_t(q,q')$.
\end{proof}

\subsection{Gromov-Hausdorff convergence}
We will now show the Gromov-Hausdorff convergence of $(X,\omega_{t_n})$ to $(B,\omega_{B,\infty})$ for a subsequence $\{t_n\}_{n=1}^\infty\rightarrow T^-$ and a metric $\omega_{B,\infty}$ that is uniformly equivalent to $\omega_B$. The proof was shown by Fu-Zhang \cite{fu-zhang17}, but we include a proof here for the convenience of the reader.

Firstly, for any $x,y\in X$, where we denote $d_{\omega_t}$ by $d_t$ and $\diam_{\omega(t)}$ by $\diam_t$ for simplicity, we have that 
\begin{align}\begin{split}
d_t(x,y)&\le \diam_t F_{\pr_B(x)} + \sqrt{C}d_0(\pr_B(x),\pr_B(y))+\diam_t F_{\pr_B(y)}\\
&\le C(T-t)^{1/5}+ \sqrt{C}d_B(\pr_B(x),\pr_B(y)).\label{triangle}
\end{split}\end{align}
In particular, there exists a uniform $C$ such that $d_t(x,y)\le C$ for any $t<T$ and for all $x,y\in X$.

Now, let $M=X\times X$. For each $k\in\mathbb{N}$, let $Q_k$ be the finite collection of centers of balls of radius $\frac{1}{k}$, measured with respect to $\pr_1^*\omega_0 + \pr_2^*\omega_0$, that covers $M$ by compactness, where $\pr_1$ and $\pr_2$ are projections onto the first and second components of $M$, respectively. We note that $Q$ is countable and dense in $X$. Take a sequence $(d_{t_n})_{n=1}^\infty$ of $d_t$ with $t_n\rightarrow T$ as $n\rightarrow\infty$. For a point $q_1\in Q$, the sequence $(d_{t_n}(q_1))_{n=1}^\infty$ is bounded and so we can find a subsequence $(t_{1,n})$ such that $d_{t_{1,n}}(q_1)$ converges in $\mathbb{R}$. We can then find a subsequence $(t_{2,n})$  of $(t_{1,n})$ such that $d_{t_{2,n}}(q_2)$ converges. Using the fact that $Q$ is countable, we proceed in this way and construct a subsequence that converges at each $q\in Q$ by taking the diagonal sequence $(d_{t_{n}})=(d_{t_{n,n}})$. We are left to show that this subsequence is uniformly Cauchy. 

Given $\varepsilon>0$, let $\delta = \frac{\varepsilon}{12C}$, $k>\frac{1}{\delta}$ and choose $T_\varepsilon$ such that $(T-T_\varepsilon)^{1/5}\le \frac{\varepsilon}{12C}$. Then for any $q=(x,y)\in X$, we have that
\begin{align*}
|d_{t_m}(q)-d_{t_n}(q)|&\le |d_{t_m}(q)-d_{t_m}(q_k)| +|d_{t_n}(q_k)-d_{t_m}(q_k)|+|d_{t_n}(q_k)-d_{t_n}(q)|
\end{align*}
where  $q_k=(x_k, y_k)\in Q_k$.
The first term can be bounded as
\begin{align*}
|d_{t_m}(q)-d_{t_m}(q_k)|&\le |d_{t_m}(x,y)-d_{t_m}(x,y_k)|+|d_{t_m}(x,y_k)-d_{t_m}(x_k, y_k)|\\
&\le d_{t_m}(y,y_k)+d_{t_m}(x,x_k)\\
&\le \frac{\varepsilon}{6}+\frac{\varepsilon}{6}=\frac{\varepsilon}{3}
\end{align*}
using Equation \eqref{triangle} and the fact that $d_0(\pr_B(x),\pr_B(x_k))\le \delta$ and $d_0(\pr_B(y),\pr_B(y_k))\le \delta$ since we chose $k>\frac{1}{\delta}$. The same holds for the third term by replacing $m$ with 
$n$. Since there are only a finite number of points in $Q_k$, choose $N$ large enough such that for $n,m>N$, 
\begin{align*}
|d_{t_n}(q_k)-d_{t_m}(q_k)|\le \frac{\varepsilon}{3}.
\end{align*}
Taking $N$ to be large enough such that $T_N\ge T_\varepsilon$, we see that $d_{t_n}$ is uniformly Cauchy and, thus, uniformly convergent. Let $d_\infty$ be the limit of this subsequence. It is straightforward to see that $d_\infty$ is continuous, non-negative symmetric and satisfies the triangle inequality. 

We have by Lemma \ref{metric_lower_bound} that $$d_\infty(x,y)\ge \sqrt{c}d_B(\pr_B(x,\pr_B(y)).$$

We also have an upper bound $$d_\infty(x,y)\le \sqrt{C}d_B(\pr_B(x),\pr_B(y)),$$ following from Equation \eqref{triangle}. Define $d_{B,\infty}(x,y) = d_\infty(x',y')$, for $x'\in F_x$ and $y'\in F_y$ which is independent of the choice of lift since for another lift $y''$ of $y$, we have
\begin{align*}
d_\infty(x',y')&\le d_\infty(x',y'')+d_\infty(y'',y')=d_\infty(x',y'')\\
d_\infty(x',y'')&\le d_\infty(x',y')+d_\infty(y',y'')=d_\infty(x',y'),
\end{align*}
which gives that $d_\infty(x',y')=d_\infty(x',y'')$. By symmetry, the same holds for lifts of $x$. Using the characterization of Gromov-Hausdorff convergence from \cite{fukaya}, the Gromov-Hausdorff distance $d_{\GH}(X,Y)$ between two metric spaces $(X,d_X)$ and $(Y,d_Y)$ is the infimum of all $\varepsilon>0$ such that the following holds. These exist, not necessarily continuous, maps $F:X\rightarrow Y$ and $G:Y\rightarrow X$ such that
\begin{align*}
|d_X(x_1)-d_Y(F(x_1),F(x_2))|&<\varepsilon \ \text{for all } x_1,x_2\in X,\\
d_X(x,G\circ F(x))&<\varepsilon \ \text{for all } x\in X,
\end{align*}
and similiarly for $Y$. Define $F:X\rightarrow B$ to be the projection map $F=\pr_B$, and let $G:B\rightarrow X$ be any map satisfying $F\circ G(b)=b$ for all $b\in B$. Then we have for any $x_1,x_2,x\in X$ that
\begin{align*}
|d_{t_n}(x_1,x_2)-d_{B,\infty}(F(x_1),F(x_2))| &= |d_{t_n}(x_1,x_2)-d_\infty(x_1,x_2)|\rightarrow 0,\\
d_{t_n}(x,G\circ F(x))&\le \diam_{t_n}F_{\pr_B(x)}\le C(T-t)^{1/5}\rightarrow 0,
\end{align*}
since $d_{t_n}\rightarrow d_\infty$ uniformly and because of Lemma \ref{fiber_convergence}. For all $b_1,b_2,b\in B$ we have 
\begin{align*}
|d_{B,\infty}(b_1,b_2)-d_{t_n}(G(b_1),G(b_2))|&=|d_\infty(G(b_1),G(b_2))|\rightarrow 0,
d_{B,\infty}(b,F\circ G(b))&=0,
\end{align*}
again by the fact that $d_{t_n}\rightarrow d_\infty$ uniformly. Thus, we have shown that as $n\rightarrow \infty$, we have that $d_{\GH}(X,d_{t_n}),(B,d_{B,\infty}))\rightarrow 0$, as desired.
\qed

\section{Generalizing to the fiber bundle case}
In this subsection, we will describe the proof of Theorem \ref{mainthm} by adapting the estimates we have obtained in the last subsection to the case of fiber bundles which can be locally trivialized over Zariski open sets. Let us define a fiber bundle to be a quadruplet $(X,B,F,\rho)$ where $B$ is the base compact manifold, $F$ is the fiber of dimension $m$, $\rho$ is the projection map from $X\rightarrow B$ such that for any $y\in B$, $\rho^{-1}(y) = F$ and on a Zariski open subset $(y\in) U\subset B$, there exists biholormorphism $\Phi$ such that the following diagram commutes
\begin{center}
\begin{tikzcd}[column sep=scriptsize]
\rho^{-1}(U) \arrow[dr, "\rho"] \arrow[rr, "\Phi"] && U\times F \arrow[dl, "\text{pr}_B"] \\
& U 
\end{tikzcd}
\end{center}
In addition to the above maps, we have $\text{pr}_F: U\times F\rightarrow F$ and $\pi: F\rightarrow \mathbb{P}^m$. For each Zariski open set $U$, let $D=B\backslash U$ and using the fact that $B$ is projective, there exists a divisor $[L]$ containing $D$ in its support. Let $s$ be a holomorphic section and a $h$ a Hermitian metric on $L$ .

The estimates on the base manifold carry over rather straightforwardly. The following lower bound holds as in the previous section:
\begin{align*}
\omega\ge C\rho^*\omega_B
\end{align*}
as well as an analogous version of Lemma \ref{base_control}:
\begin{align*}
\omega|_H\le \frac{C}{\rho^*|s|_h^{2\alpha}}(\rho^*\omega_B)|_H,
\end{align*}
which follows straightforwardly using the fact that 
\begin{align}
\Delta\log(\rho^* |s|_h^{2\alpha})\le C\tr_\omega \rho^*\omega_B\le C\label{section_bound}
\end{align} 
where the first inequality uses the fact that $\sqrt{-1}\partial\bar{\partial}\log(\rho^*|s(x_0)|^{2\alpha}_h) = \sqrt{-1}\partial\bar{\partial}\log(|s(b_0)|^{2\alpha}_h)$ is some fixed quantity on $B$ and the second inequality follows from Lemma \ref{metric_lower_bound}.
In the fiber direction, the estimates for the diameter bound need only be obtained over subsets of a finite number of Zariski open sets covering the base manifold by compactness. Hence, it suffices to prove the estimates on a subset of a Zariski open set trivializing the fiber. Let us denote 
\begin{align*}
U_{1/2}:=\{y\in U : |s(y)|_h^2>1/2\}.
\end{align*}
We will now describe how to show the analogous estimates on each $U_{1/2}$. As before, let $\pi: F\rightarrow \mathbb{P}^m$ be our blow-up map. 
It can be shown that the particular holomorphic vector fields we chose in the previous section on $\mathbb{P}^m$ indeed extend to smooth holomorphic vector fields on $\rho^{-1}(U_{1/2})$ by pulling back by $\Phi^*\pr^*_F\pi^*$. In order to make the quantity to which we apply the maximum principle global, we can consider $|\Phi^*\pr*_F\pi^* V|^2_\omega|s|^{2\alpha}_h$ for each vector field $V$ on $\mathbb{P}^m$ and some $\alpha>0$. The analogous estimates on the fiber now will be of the form:
\begin{align*}
\omega|_F \le \frac{C}{\rho^*|s|_h^{2\alpha} r^2} (\Phi^*\pr^*_F \pi^*\omega_{\mathbb{P}^m})|_F
\end{align*}
for some $\alpha>0$.
The quantities we will use in the maximum principle arguments will simply have an additional $\log(\rho^*|s|^{2\alpha}_h)$ term which we can deal with as in Equation \eqref{section_bound}. 

Now, since $T<\infty$ and since $|s|^2_h>1/2$ on $U_{1/2}$, we have that on each $U_{1/2}$ that 
\begin{align*}
\omega|_F \le \frac{C}{r^2}(\Phi^* \pr^*_F \pi^* \omega_{\mathbb{P}^m})|_F
\end{align*}

We remark that we are using the fact that the fiber bundles trivialize over one Zariski open set when we assume the existence of a holomorphic section $s$ vanishing outside of $U$.

\section{Acknowledgements}
The author would like to first and foremost thank her thesis advisor Ben Weinkove for guiding her towards this problem and for his advice and encouragement. The author would also like to thank Gregory Edwards, John Lesieutre, Nicholas McCleerey, Mihnea Popa and Valentino Tosatti and for many helpful discussions around ideas related to this paper.

\bibliography{bib_for_all}

\providecommand{\bysame}{\leavevmode\hbox to3em{\hrulefill}\thinspace}
\providecommand{\MR}{\relax\ifhmode\unskip\space\fi MR }
\providecommand{\MRhref}[2]{%
  \href{http://www.ams.org/mathscinet-getitem?mr=#1}{#2}
}
\providecommand{\href}[2]{#2}
\begin{thebibliography}{10}

\bibitem{aubin78}
T.~Aubin, \emph{{Equations du type Monge-Amp\`ere sur les varietes
  k\"ahleriennes compactes}}, Bull. Sci. Math. (2) \textbf{102} (1978), no.~1,
  63–95.

\bibitem{cao85}
H.-D. Cao, \emph{{Deformation of K{\"a}hler metrics to K{\"a}hler-Einstein
  metrics on compact K{\"a}hler manifolds.}}, Invention. Math. \textbf{81}
  (1985), 359--372.

\bibitem{chen-tian06}
X.~X. Chen and G.~Tian, \emph{{Ricci flow on K{\"a}hler-Einstein manifolds}},
  Duke Math. J. \textbf{131} (2006), no.~1, 17--73.

\bibitem{chen-tian02}
X.X. Chen and G.~Tian, \emph{{Ricci flow on K{\"a}hler-Einstein surfaces}},
  Inventiones mathematicae \textbf{147} (2002), no.~3, 487--544.

\bibitem{fik03}
M.~Feldman, T.~Ilmanen, and D.~Knopf, \emph{{Rotationally symmetric shrinking
  and expanding graduent K{\"a}hler-Ricci solitons}}, Journal of Differential
  Geometry \textbf{65} (2003), 169--209.

\bibitem{fong14}
F.~T.-H. Fong, \emph{{K\"ahler-Ricci flow on projective bundles over
  K{\"a}hler-Einstein manifolds}}, Trans. Amer. Math. Soc. \textbf{366} (2014),
  563--589.

\bibitem{fong-zhang15}
F.~T-H. Fong and Z.~Zhang, \emph{{The collapsing rate of the K{\"a}hler-Ricci
  flow with regular infinite-time singularity}}, Journal für die Reine und
  Angewandte Mathematik \textbf{703} (2015), no.~7, 95--113.

\bibitem{fu-zhang17}
X.~Fu and S.~Zhang, \emph{{K{\"a}hler-Ricci flow on Fano bundles}},
  Mathematische Zeitschrift \textbf{286(3)} (2017), 1605--1626.

\bibitem{fukaya}
K.~Fukaya, \emph{{Theory of convergence for Riemannian orbifolds}}, Japanese
  Journal of Mathematics (New Series) \textbf{12} (1986), no.~1, 121--160.

\bibitem{gill14}
M.~Gill, \emph{{Collapsing of products along the K{\"a}hler-Ricci flow}},
  Trans. Amer. Math. Soc. \textbf{366} (2014), 3907--3924.

\bibitem{guo17}
B.~Guo, \emph{{On the K{\"a}hler Ricci flow on projective manifolds of general
  type}}, International Mathematics Research Notices \textbf{2017} (2017),
  2139--2171.

\bibitem{gsw16}
B.~Guo, J.~Song, and B.~Weinkove, \emph{{Geometric convergence of the
  K{\"a}hler-Ricci flow on complex surfaces of general type}}, International
  Mathematics Research Notices \textbf{2016} (2016), no.~18, 5652--5669.

\bibitem{hamilton82}
R.~S. Hamilton, \emph{{Three-manifolds with positive Ricci curvature}}, J.
  Differential Geom. \textbf{17} (1982), no.~2, 255--306.

\bibitem{perelman1}
G.~Perelman, \emph{{The entropy formula for the Ricci flow and its geometric
  applications}}, ArXiv Mathematics e-prints (2002).

\bibitem{perelman3}
\bysame, \emph{{Finite extinction time for the solutions to the Ricci flow on
  certain three-manifolds}}, ArXiv Mathematics e-prints (2003).

\bibitem{perelman2}
\bysame, \emph{{Ricci flow with surgery on three-manifolds}}, ArXiv Mathematics
  e-prints (2003).

\bibitem{pssw09}
D.~H. Phong, J.~Song, J.~Sturm, and B.~Weinkove, \emph{{The K{\"a}hler-Ricci
  flow and the $\bar{\partial}$ operator on vector fields}}, J. Differential
  Geom. \textbf{81} (2009), no.~3, 631--647.

\bibitem{pssw11}
{D. H.} Phong, J.~Song, J.~Sturm, and B.~Weinkove, \emph{{On the convergence of
  the modified K{\"a}hler-Ricci flow and solitons}}, Commentarii Mathematici
  Helvetici \textbf{86} (2011), no.~1, 91--112 (English (US)).

\bibitem{phong-sturm06}
D.~H. Phong and J.~Sturm, \emph{{On stability and the convergence of the
  K{\"a}hler-Ricci flow}}, J. Differential Geom. \textbf{72} (2006), no.~1,
  149--168.

\bibitem{pssw08}
D.H. Phong, J.~Song, J.~Sturm, and B.~Weinkove, \emph{{The K{\"a}hler-Ricci
  flow with positive bisectional curvature}}, Inventiones mathematicae
  \textbf{173} (2008), no.~3, 651--665.

\bibitem{sesum-tian08}
N.~Sesum and G.~Tian, \emph{{Bounding scalar curvature and diameter along the
  K{\"a}hler-Ricci flow (after Perelman)}}, Journal of the Institute of
  Mathematics of Jussieu \textbf{7} (2008), no.~3, 575–587.

\bibitem{song14}
J.~Song, \emph{{Finite-time extinction of the K{\"a}hler-Ricci flow}}, Math.
  Res. Lett. \textbf{21} (2014), 1435--1449.

\bibitem{ssw13}
J.~Song, G.~Szekelyhidi, and B.~Weinkove, \emph{{The K\"ahler-Ricci flow on
  projective bundles}}, IMRN \textbf{2} (2013), 243--257.

\bibitem{song-tian07}
J.~Song and G.~Tian, \emph{{The K{\"a}hler--Ricci flow on surfaces of positive
  Kodaira dimension}}, Inventiones mathematicae \textbf{170} (2007), no.~3,
  609--653.

\bibitem{st17}
\bysame, \emph{{The K\"ahler-Ricci flow through singularities}}, Invent. Math.
  \textbf{207} (2017), 519--595.

\bibitem{sw11}
J.~Song and B.~Weinkove, \emph{{The K\"ahler-Ricci flow on Hirzebruch
  surfaces}}, J. Reine Angew. Math. (2011).

\bibitem{song-weinkove12}
\bysame, \emph{{Lecture Notes on K{\"a}hler-Ricci flow}}, arXiv:1212.3653
  (2012).

\bibitem{song-weinkove13}
\bysame, \emph{{Contracting exceptional divisors by the K{\"a}hler-Ricci
  flow}}, Duke Math Journal \textbf{162} (2013), 367--415.

\bibitem{szekelyhidi10}
G~Szekelyhidi, \emph{{The K{\"a}hler-Ricci flow and k-polystability}}, American
  Journal of Mathematics \textbf{132} (2010), no.~4, 1077--1090.

\bibitem{tian08}
G.~Tian, \emph{{New results and problems on K{\"a}hler-Ricci flow}},
  {G\'eom\'etrie diff\'erentielle, physique math\'ematique, math\'ematiques et
  soci\'et\'e. II. Asterisque} \textbf{322} (2008), 71--92.

\bibitem{tian10}
\bysame, \emph{{Finite-time singularity of K{\"a}hler-Ricci flow}}, Discrete
  Contin. Dyn. Syst. \textbf{28} (2010), 1137--1150.

\bibitem{tzzz13}
G.~Tian, S.J. Zhang, Z.L. Zhang, and X.H. Zhu, \emph{{Perelman's entropy and
  K{\"a}hler-Ricci flow on a Fano manifold}}, Trans. Amer. Math. Soc.
  \textbf{365} (2013), 6669--6695.

\bibitem{tian-zhang06}
G.~Tian and Z.~Zhang, \emph{{On the K{\"a}hler-Ricci flow on projective
  manifolds of general type}}, Chinese Annals of Mathematics, Series B
  \textbf{27} (2006), no.~2, 179--192.

\bibitem{tian-zhang-fano}
\bysame, \emph{{Regularity of K\"ahler-Ricci flows on Fano manifolds}}, Acta.
  Math. (2007), 127--176.

\bibitem{tian-zhang16}
\bysame, \emph{{Convergence of K{\"a}hler-Ricci flow on lower-dimensional
  algebraic manifolds of general type}}, International Mathematics Research
  Notices \textbf{2016} (2016), no.~21, 6493--6511.

\bibitem{tian-zhu07}
G.~Tian and X.H. Zhu, \emph{{Convergence of K{\"a}hler-Ricci flow}}, J. Amer.
  Math. Soc. \textbf{20} (2007), 675--699.

\bibitem{tian-zhu13}
\bysame, \emph{{Convergence of the K{\"a}hler-Ricci flow on Fano manifolds}},
  J. Reine Angew. Math. \textbf{678} (2013), 223--245.

\bibitem{twy14}
V.~{Tosatti}, B.~{Weinkove}, and X.~{Yang}, \emph{{The K{\"a}hler-Ricci flow,
  Ricci-flat metrics and collapsing limits}}, Amer. J. Math. \textbf{140}
  (2018), no.~3, 653–698.

\bibitem{tsuji88}
H.~Tsuji, \emph{{Existence and Degeneration of Kähler-Einstein Metrics on
  Minimal Algebraic Varieties of General Type.}}, Mathematische Annalen
  \textbf{281} (1988), no.~1, 123--134.

\bibitem{yau-schwarz}
S.-T. Yau, \emph{{A general Schwarz lemma for K\"ahler manifolds}}, Amer. J.
  Math. \textbf{100} (1978), no.~1, 197–203.

\bibitem{yau78}
\bysame, \emph{{On the Ricci curvature of a compact K{\"a}hler manifold and the
  complex Monge-Amp\`ere equation, I}}, Comm. Pure Appl. Math. \textbf{31}
  (1978), 339–411.

\bibitem{zhang06}
Z.~Zhang, \emph{{On degenerate Monge-Amp\`ere equations over closed K\"ahler
  manifolds}}, Int. Math. Res. Not. (2006).

\end{thebibliography}
\end{document}